\documentclass[reqno]{amsart}

\usepackage{amsmath}
\usepackage{amsfonts}
\usepackage{amssymb,enumerate}
\usepackage{amsthm}
\usepackage[all]{xy}
\usepackage{rotating}
\usepackage[notref,notcite]{}
\usepackage{hyperref}
\usepackage{todonotes}
\usepackage{color}

\theoremstyle{plain}
\newtheorem{lem}{Lemma}[section]
\newtheorem{cor}[lem]{Corollary}
\newtheorem{prop}[lem]{Proposition}
\newtheorem{thm}[lem]{Theorem}

\theoremstyle{definition}
\newtheorem{chunk}[lem]{\hspace*{-1.065ex}\bf}
\newtheorem{defn}[lem]{Definition}
\newtheorem{ex}[lem]{Example}

\newtheorem{que}[lem]{Question}

\newtheorem{disc}[lem]{Remark}
\newtheorem{rmk}[lem]{Remark}

\newtheorem{subprops}{}[lem]

\newcommand{\cat}[1]{\mathcal{#1}}

\newcommand{\catd}{\cat{D}}

\newcommand{\cata}{\cat{A}}

\newcommand{\catb}{\cat{B}}

\newcommand{\catac}{\cat{A}_C}
\newcommand{\catab}{\cat{A}_B}
\newcommand{\catbc}{\cat{B}_C}

\newcommand{\catacd}{\cat{A}_{\da{C}}}

\newcommand{\pd}{\operatorname{pd}}	
\newcommand{\gdim}{\mathrm{G}\text{-}\!\dim}	
\newcommand{\gkdim}[1]{\mathrm{G}_{#1}\text{-}\!\dim}	
\newcommand{\gcdim}{\gkdim{C}}

\newcommand{\fd}{\operatorname{fd}}

\newcommand{\depth}{\operatorname{depth}}

\newcommand{\len}{\operatorname{len}}

\newcommand{\HH}{\operatorname{H}}
\newcommand{\Hom}{\operatorname{Hom}}	
\newcommand{\coker}{\operatorname{Coker}}

\newcommand{\da}[1]{#1^{\dagger}}

\newcommand{\ideal}[1]{\mathfrak{#1}}
\newcommand{\m}{\ideal{m}}
\newcommand{\n}{\ideal{n}}

\newcommand{\fm}{\ideal{m}}
\newcommand{\fn}{\ideal{n}}

\newcommand{\comp}[1]{\widehat{#1}}

\newcommand{\bbz}{\mathbb{Z}}

\newcommand{\xra}{\xrightarrow}

\newcommand{\vf}{\varphi}

\newcommand{\x}{\mathbf{x}}

\renewcommand{\geq}{\geqslant}
\renewcommand{\leq}{\leqslant}

\newcommand{\Ext}[4][R]{\operatorname{Ext}_{#1}^{#2}(#3,#4)}	
\newcommand{\Rhom}[3][R]{\mathbf{R}\!\operatorname{Hom}_{#1}(#2,#3)}	
\newcommand{\Lotimes}[3][R]{#2\otimes^{\mathbf{L}}_{#1}#3}
\newcommand{\Otimes}[3][R]{#2\otimes_{#1}#3}
\renewcommand{\Hom}[3][R]{\operatorname{Hom}_{#1}(#2,#3)}	
\newcommand{\Tor}[4][R]{\operatorname{Tor}^{#1}_{#2}(#3,#4)}

\newcommand{\ssm}{\smallsetminus}

\newcommand{\Hod}{\operatorname{H-dim}}

\newcommand{\catdb}{\cat{D}_{\text{b}}}
\newcommand{\catdfb}{\cat{D}^{\text{f}}_{\text{b}}}
\newcommand{\catdf}{\cat{D}^{\text{f}}}

\numberwithin{equation}{lem}

\begin{document}

\author{Olgur Celikbas}
\address{University of Connecticut, Department of Mathematics, Storrs, CT 06269--3009 USA} 
\email{olgur.celikbas@uconn.edu}

\author{Sean Sather-Wagstaff}

\address{NDSU Department of Mathematics \# 2750,
PO Box 6050,
Fargo, ND 58108-6050
USA}

\curraddr{Department of Mathematical Sciences,
Clemson University,
O-110 Martin Hall, Box 340975, Clemson, S.C. 29634
USA}

\email{ssather@clemson.edu}

\urladdr{http://people.clemson.edu/\~{}ssather/}

\thanks{Sean Sather-Wagstaff was supported in part by a grant from the NSA}

\title[Testing for the Gorenstein property]
{Testing for the Gorenstein property}

\date{\today}


\begin{abstract}
We answer a  question of Celikbas, Dao, and Takahashi by establishing the following characterization of Gorenstein rings: a commutative noetherian local ring $(R,\mathfrak m)$ is Gorenstein if and only if it admits an integrally closed $\mathfrak m$-primary ideal of finite Gorenstein dimension. This is accomplished through a detailed study of certain test complexes. Along the way we construct such a test complex that detect finiteness of Gorenstein dimension, but not that of projective dimension.
\end{abstract}

\maketitle

\section{Introduction}

Throughout this paper $R$ denotes a commutative noetherian local ring with unique maximal ideal $\fm$ and residue field $k$.

\

A celebrated theorem of Auslander, Buchsbaum, and Serre~\cite{auslander:hdlr,serre:sldhdaedmn} tells us that  $R$ is regular if and only if $k$ has finite projective dimension. 
Burch \cite[p. 947, Corollary 3]{B} extended this by proving that $R$ is regular if and only if $\pd(R/I)<\infty$ for some integrally closed $\fm$-primary ideal $I$ of $R$. 

Auslander and Bridger~\cite{auslander:adgeteac,AuBr} introduced the G-dimension as a generalization of projective dimension. 
(See~\ref{ch150423a} for the definition.)
Analogous to the regular setting, the finiteness of $\gdim_R(k)$ characterizes the Gorensteinness of $R$.
In our local setting, Goto and Hayasaka \cite{GH} studied Gorenstein dimension of integrally closed $\fm$-primary ideals and,  analogous to Burch's result, established the following; see the question of K. Yoshida stated in the discussion following \cite[(1.1)]{GH}.

\begin{chunk}[\protect{\cite[(1.1)]{GH}}] \label{GoHa} Let  $I$ be an integrally closed $\fm$-primary ideal of $R$. Assume $I$ contains a non-zerodivisor of $R$, or $R$ satisfies Serre's condition $(S_{1})$. Then $R$ is Gorenstein if and only if $\gdim_R(R/I)<\infty$. 
\end{chunk}

Our aim in this paper is to remove the hypothesis ``$I$ contains a non-zerodivisor of $R$, or $R$ satisfies Serre's condition $(S_{1})$" from~\ref{GoHa}. We accomplish this in the following result and hence obtain a complete generalization of Burch's aforementioned result; see also Corollary \ref{cor2z} for a further generalization. 

\

\noindent
\textbf{Theorem~\ref{cor2}.}
\emph{Let $I\subseteq R$ be an integrally closed ideal with $\depth(R/I)=0$, e.g., such that $I$ is $\fm$-primary. Then $R$ is Gorenstein if and only if $\gdim_R(R/I)<\infty$.}

\

Our argument 
is quite different from that of Goto and Hayasaka \cite{GH}
since it uses $\gdim$-test complexes.
For this part of the introduction, we focus on the case of $\Hod$-test modules, defined next.
Note that the  $\pd$-test modules
(i.e., the case where $\Hod=\pd$) 
are from~\cite{CDtest}.

\begin{defn} \label{moduletest} 
Let $\Hod$ denote either projective dimension $\pd$ or G-dimension $\gdim$.
Let $M$ be a finitely generated $R$-module.
Then $M$ is an \emph{$\Hod$-test module} over $R$ if the following condition holds for all finitely generated $R$-modules $N$:
If $\Tor iMN=0$ for all $i\gg 0$, then $\Hod_R(N)<\infty$.  
\end{defn}

It is straightforward to show that if $M$ is a $\pd$-test $R$-module, then it is also $\gdim$-test.
Example~\ref{ex140816a} shows that the converse of this statement fails in general.

As part of our proof of Theorem~\ref{cor2}, we also answer the following questions;
see
Corollaries~\ref{cor140825b}\eqref{cor140825b2} and~\ref{cor140825a}.

\begin{que}\label{q1} 
Let $M$ be a  $\pd$-test module over $R$. 
\begin{enumerate}[(a)]
\item\label{item140831b}
Must $\widehat{M}$ be a $\pd$-test module over $\widehat{R}$?
\item\label{item140831a}
(\cite[(3.5)]{CDtest}) If  $\gdim_R(M)<\infty$, must $R$ be Gorenstein? 
\end{enumerate}
\end{que}

Affirmative answers to Question~\ref{q1}\eqref{item140831a} under additional hypotheses are in~\cite[(1.3)]{CDtest} and~\cite[(2.15)]{Testrigid}.
Also, Majadas~\cite{JM} gives an affirmative answer 
to a version of Question~\ref{q1}\eqref{item140831b}
that uses a  more restrictive version of test modules.

Theorem~\ref{cor2} follows from the next, significantly stronger result.

\

\noindent
\textbf{Theorem~\ref{mainobscor}.} 
\emph{Let $M$ be a  $\gdim$-test $R$-module such that $\Ext iMR=0$ for $i\gg 0$. Then $R$ is Gorenstein.}

\

In turn, this follows from the much more general Theorem~\ref{mainobs} and Corollary~\ref{mainobsx}, which are results for detecting dualizing complexes. 

We conclude this introduction by summarizing the contents of this paper.
Section~\ref{sec140831} consists of background material for use throughout the paper, and contains some technical lemmas for later use.
In Section~\ref{sec150429a}, we develop foundational properties of various $\Hod$-test objects, and answer Question~\ref{q1}.
And Section~\ref{sec150429b} contains the theorems highlighted above.

\section{Derived Categories and Semidualizing Complexes}
\label{sec140831}

\begin{chunk}
Throughout this paper we work in the derived category $\catd(R)$ whose objects are the chain complexes of $R$-modules, i.e., the
$R$-complexes 
$X$ with homological differential $\partial^X_i\colon X_i\to X_{i-1}$.
References for this include~\cite{Larsbook2,Hartshorne, verdier:cd, verdier:1}. Our notation is consistent with~\cite{LarsSD}. 
In particular, $\Rhom XY$ and $\Lotimes XY$ are the derived Hom-complex and derived tensor product of two $R$-complexes $X$ and $Y$.
Isomorphisms in $\catd(R)$ are identified by the symbol $\simeq$.
The projective dimension and flat dimension of an $R$-complex $X\in\catdb(R)$ are denoted $\pd_R(X)$ and $\fd_R(X)$.

The subcategory of $\catd(R)$ consisting of homologically bounded $R$-complexes
(i.e., complexes $X$ such that $\HH_i(X)=0$ for $|i|\gg 0$) is $\catdb(R)$.
The subcategory of $\catd(R)$ consisting of homologically finite $R$-complexes
(i.e., complexes $X$ such that $\HH(X):=\oplus_{i\in\bbz}\HH_i(X)$ is finitely generated) is denoted $\catdfb(R)$.
\end{chunk}

\begin{chunk}\label{ch150423c}
A homologically finite $R$-complex $C$ is \emph{semidualzing} if the natural morphism $R\to\Rhom CC$ in $\catd(R)$ is an isomorphism.
For example, an $R$-module is semidualizing if and only if $\Hom CC\cong R$ and $\Ext iCC=0$ for $i\geq 1$.
In particular, $R$ is a semidualizing $R$-module.
A \emph{dualizing $R$-complex} is a semidualizing $R$-complex of finite injective dimension.

\begin{subprops}\label{ch150423c1}
If $R$ is a homomorphic image of a local Gorenstein ring $Q$, then $R$ has a dualizing complex, by~\cite[V.10]{Hartshorne}.
(The converse holds by work of Kawasaki~\cite{Kawasaki}.)
In particular, the Cohen Structure Theorem shows that the completion $\comp R$ has a dualizing complex.
When $R$ has a dualizing complex $D$, and $C$ is a semidualizing $R$-complex, the dual $\Rhom CD$ is also semidualizing over $R$, by~\cite[(2.12)]{LarsSD}.
\end{subprops}

\begin{subprops}\label{ch150423c2}
Let $\vf\colon R\to S$ be a flat local ring homomorphism,
and let $C$ be a semidualizing $R$-complex.
Then the $S$-complex $\Lotimes SC$ is semidualizing, by~\cite[(5.6)]{LarsSD}.
If  the closed fibre $S/\m S$ is Gorenstein and $R$ has a dualizing complex $D^R$, then
$D^S:=\Lotimes S{D^R}$ is dualizing for $S$ by~\cite[(5.1)]{avramov:lgh}.
\end{subprops}
\end{chunk}

Dualizing complexes were introduced by Grothendieck and Harshorne~\cite{Hartshorne}.
The more general semidualizing complexes originated in special cases, e.g., in~\cite{avramov:rhafgd, foxby:gmarm, Golod, vasconcelos:dtmc},
with the general version premiering in~\cite{LarsSD}.
The notion of G-dimension, summarized next, started with the work of Auslander and Bridger~\cite{auslander:adgeteac,AuBr} for modules.
Foxby and Yassemi~\cite{yassemi:gd} recognized the connection with derived reflexivity,
with the general situation given in~\cite{LarsSD}.

\begin{chunk}\label{ch150423a}
Let $C$ be a semidualizing $R$-complex and $X\in\catdfb(R)$. Write $\gcdim_R(X)<\infty$ when
$X$ is ``derived $C$-reflexive'', i.e., when
$\Rhom XC\in\catdb(R)$ and the natural morphism $X\to\Rhom{\Rhom XC}C$ in $\catd(R)$ is an isomorphism.
In the case $C=R$, we write $\gdim_R(X)<\infty$ instead of $\gkdim R_R(X)<\infty$.

\begin{subprops}\label{ch150423a1}
The complex $C$ is dualizing if and only if every $R$-complex in $\catdfb(R)$ is derived $C$-reflexive, by~\cite[(8.4)]{LarsSD}.
In particular, $R$ is Gorenstein if and only if every $R$-complex $X\in\catdfb(R)$ has $\gdim_R(X)<\infty$.
\end{subprops}

\begin{subprops}\label{ch150423a3}
Let $R\to S$ be a flat local ring homormorphism. 
Given an $R$-complex $X\in\catdfb(R)$, one has
$\gcdim_R(X)<\infty$ if and only if $\gkdim{\Lotimes SC}_S(\Lotimes SX)<\infty$, by~\cite[(5.10)]{LarsSD}; see~\ref{ch150423c2}.
\end{subprops}
\end{chunk}

Auslander and Bass classes, defined next, arrived in special cases in~\cite{avramov:rhafgd,foxby:gmarm},
again with the general case described in~\cite{LarsSD}.

\begin{chunk}\label{ch150423b}
Let $C$ be a semidualizing $R$-complex.
The \emph{Auslander class}  $\catac(R)$ consists of the $R$-complexes $X\in\catdb(R)$
such that $\Lotimes CX\in\catdb(R)$ and the natural  morphism $\gamma^C_X\colon X\to\Rhom C{\Lotimes CX}$ in $\catd(R)$ is an isomorphism.
The \emph{Bass class} $\catbc(R)$ consists of all the $R$-complexes $X\in\catdb(R)$
such that $\Rhom CX\in\catdb(R)$ and such that the natural  morphism $\xi^C_X\colon\Lotimes C{\Rhom CX}\to X$ in $\catd(R)$ is an isomorphism.

\begin{subprops}\label{ch150423b1}
When $R$ has a dualizing complex $D$, given an $R$-complex $X\in\catdfb(R)$, one has
$\gcdim_R(X)<\infty$ if and only if $X\in\cata_{\Rhom CD}(R)$, by~\cite[(4.7)]{LarsSD}; this uses~\ref{ch150423c1} and~\ref{ch150423a1}, which imply that
$\Rhom CD$ is semidualizing and $C\simeq\Rhom{\Rhom CD}D$.
\end{subprops}

\begin{subprops}\label{ch150423b2}
Let $R\to S$ be a flat local ring homormorphism. 
Given an $S$-complex $X$, one has
$X\in\catac(R)$ if and only if $X\in\cata_{\Lotimes SC}(S)$, by~\cite[(5.3.a)]{LarsSD}.
\end{subprops}
\end{chunk}

The following two lemmas are proved like~\cite[(4.4)]{frankild:rrhffd} and~\cite[(7.3)]{iyengar:golh}, respectively.

\begin{lem}\label{fact140901a}
Let $X,P\in\catdfb(R)$ such that $P\not\simeq 0$ and $\pd_R(P)<\infty$.
Let $C$ be a semidualizing $R$-complex.
\begin{enumerate}[\rm(a)]
\item\label{fact140901a1}
The following conditions are equivalent:
\begin{enumerate}[\rm(i)]
\item\label{fact140901a1a}
$X\in\cata_C(R)$, 
\item\label{fact140901a1b}
$\Lotimes PX\in\catac(R)$, and
\item\label{fact140901a1c}
$\Rhom PX\in\catac(R)$.
\end{enumerate}
\item\label{fact140901a2}
The following conditions are equivalent:
\begin{enumerate}[\rm(i)]
\item\label{fact140901a2a}
$X\in\catb_C(R)$, 
\item\label{fact140901a2b}
$\Lotimes PX\in\catbc(R)$, and
\item\label{fact140901a2c}
$\Rhom PX\in\catbc(R)$.
\end{enumerate}
\end{enumerate}
\end{lem}

\begin{lem}\label{lem150421a}
Let $R\to S$ be a flat local ring homomorphism
such that $S/\m S$ is Gorenstein.
Let $X\in\catdfb(S)$ such that each homology module $\HH_i(X)$ is finitely generated over $R$.
\begin{enumerate}[\rm(a)]
\item \label{lem150421a1}
One has $\gcdim_R(X)<\infty$ if and only if $\gkdim{\Lotimes SC}_S(X)<\infty$.
\item \label{lem150421a2}
One has $\gdim_R(X)<\infty$ if and only if $\gdim_S(X)<\infty$.
\end{enumerate}
\end{lem}

The next result, essentially from~\cite[Theorem 5(ii)]{B}, is key for  Theorem~\ref{cor2}.

\begin{lem}\label{lem150424a}
Let $I$ be an integrally closed ideal such that $\depth(R/I)=0$, and let $M$ be a finitely generated $R$-module. 
If $\Tor{i}{R/I}M=0=\Tor{i+1}{R/I}M$ for some $i\geq 1$, then $\pd_R(M)\leq i$.
In particular, $R/I$ is a $\pd$-test module over $R$.
\end{lem}

\begin{proof}
Assume that $\Tor{i}{R/I}M=0=\Tor{i+1}{R/I}M$, and suppose that $\pd_R(M)> i$.
The Auslander-Buchsbaum-Serre Theorem states that $k$ is a $\pd$-test $R$-module, so we assume
without loss of generality that $I\subsetneq\m$. Hence, we have $I\colon\m\neq R$ so $I\colon\m\subseteq\m$.
From~\cite[Theorem 5(ii)]{B}, we conclude that $\m I=\m(I\colon\m)$.

Claim: $I\colon \m=I$. One containment $(\supseteq)$ is standard. For the reverse containment, let
$r\in I\colon\m\subseteq\m$. To show that $r$ is in $I$, it suffices to show that $r$ is integral over $I$,
since $I$ is integrally closed. To this end, we use the ``determinantal trick'' from~\cite[(1.1.8)]{CS}:
it suffices to show that (1) we have $r\m\subseteq I\m$, and (2) whenever $a\m=0$ for some $a\in R$, we have $ar=0$.
For (1), since $r$ is in $I\colon\m$, we have $r\m\subseteq (I\colon\m)\m=I\m$, as desired.
For (2), if $a\m=0$, then the fact that $r$ is in $\m$ implies that $ar=0$, as desired.
This completes the proof of the claim.

Now, the fact that $\depth(R/I)=0$ implies that there is an element $x\in R\ssm I$ such that $x\m\subseteq I$. 
In other words, $x\in(I\colon\m)\ssm I$, contradicting the above claim.
\end{proof}

\begin{chunk}\label{ch150424a}
Let $I$ be an integrally closed  ideal of $R$. 
If one assumes that $I$ is $\m$-primary (stronger than the assumption $\depth(R/I)=0$ from Lemma~\ref{lem150424a}) then one gets the following very strong conclusion.
Given a finitely generated $R$-module $M$, if $\Tor i{R/I}M=0$ for some $i\geq 1$, then $\pd_R(M)<i$, by~\cite[(3.3)]{CHKV}.
\end{chunk}

\section{$\Hod$-Test Complexes}
\label{sec150429a}

In this section, let $C$ be a semidualizing $R$-complex.

\

We now introduce the main object of study for this paper.

\begin{defn} \label{Gtest} 
Let $M\in\catdfb(R)$, and let $\Hod$ denote either projective dimension $\pd$ or $\text{G}_C$-dimension $\gcdim$.
Then $M$ is an \emph{$\Hod$-test complex} over $R$ if the following condition holds for all  $N\in\catdfb(R)$:
If $\Tor iMN=0$ for all $i\gg 0$, i.e., if $\Lotimes MN\in\catdb(R)$, then $\Hod_R(N)<\infty$.  
\end{defn}

\begin{chunk} \label{obs1} 
Let $M$ be an $R$-module.
A standard truncation argument shows that $M$ is a $\Hod$-test module if and only if it is a $\Hod$-test complex; see~\cite[Proof of (3.2)]{CDtest}.
\end{chunk}

\begin{chunk}\label{obs140825a}
Examples of $\pd$-test modules are given in~\ref{ch150424a} and Lemma~\ref{lem150424a}.
Note that this includes the standard example $k=R/\m$.
See also~\cite[Appendix A]{Testrigid}.

\begin{subprops}\label{obs140825a1}
Given an $R$-complex $X\in\catdfb(R)$,
if $\pd_R(X)$ is finite, then so is $\gcdim_R(X)$, by~\cite[(2.9)]{LarsSD}.
Thus, if $M$ is a  $\pd$-test complex, then it is also a  $\gcdim$-test, in particular, $M$ is a $\gdim$-test complex.
\end{subprops}

\begin{subprops}\label{obs140825a2zz}
If $R$ is G-regular (i.e., if every $R$-module of finite G-dimension has finite projective dimension), then
the $\pd$-test complexes and $\gdim$-test complexes over $R$ are the same. 
Examples of G-regular rings include regular rings and Cohen-Macaulay rings of minimal multiplicity; see~\cite{MR2473342}.
\end{subprops}

\begin{subprops}\label{obs140825a2}
Examples of $\gcdim$ test modules, e.g., of $\gdim$ test modules, that are not $\pd$ test modules are more mysterious.
See Example~\ref{ex140816a} for a non-trivial example. 
\end{subprops}

\begin{subprops}\label{obs140825a3}
Assume that $R$ has a dualizing complex $D$.
A natural candidate for a $\gcdim$-test complex is $\da C=\Rhom CD$.
Indeed, if $\gcdim_R(X)<\infty$, then $X\in\catacd(R)$ by~\ref{ch150423b1},
so by definition we have $\Lotimes{\da C}X\in\catdb(R)$.
In particular, a natural candidate for a $\gdim$-test complex is $D$; see, e.g., Corollary~\ref{cor140901a}.

However, $D$ can fail to be a $\gdim$-test complex.
Indeed, Jorgensen and \c Sega~\cite[(1.7)]{DJLS} construct an artinian local ring $R$ with a finitely generated module $L$
that satisfies $\Ext iLR=0$ for all $i\geq 1$ and $\gdim_R(L)=\infty$.
Since $R$ is local and artinian, it has a dualizing complex, namely $D=E_R(k)$ the injective hull of the residue field $k$.
We claim that $\Tor iDL=0$ for  $i\geq 1$.
(This shows that $D$ is not  $\gdim$-test over $R$.)
To this end, recall the following for $i\geq 1$:
$$\Hom{\Tor iDL}D\cong\Ext iLR=0.$$
The fact that $D$ is faithfully injective implies that $\Tor iDL=0$ for  $i\geq 1$.
\end{subprops}
\end{chunk}

\begin{chunk}\label{obs150421a}
The ring $R$ is regular if and only if every $X\in\catdfb(R)$ has $\pd_R(X)<\infty$.
Hence, the trivial complex $0$ is a $\pd$-test complex if and only if $R$ is regular,
equivalently, if and only if $R$ has a $\pd$-test complex of finite projective dimension.
Similarly, 
if $C$ is a semidualizing $R$-complex, then
$0$ is a  $\gcdim$-test complex if and only if $C$ is dualizing, 
equivalently, if and only if $R$ has a $\gcdim$-test complex of finite projective dimension;
see~\ref{ch150423a1}.
In particular, $0$ is a  $\gdim$-test complex if and only if $R$ is Gorenstein,
equivalently, if and only if $R$ has a $\gdim$-test complex of finite projective dimension.
\end{chunk}

We continue with a discussion of ascent and descent of test complexes.

\begin{thm}\label{prop140825a}
Let $R\xra\vf S$ be a  flat local ring homomorphism, and let $M\in\catdfb(R)$.
\begin{enumerate}[\rm(a)]
\item\label{prop140825a2}
If $\Lotimes SM$ is $\gkdim{\Lotimes SC}$-test  over $S$, then $M$ is   $\gcdim$-test  over $R$.
\item\label{prop140825a2'}
If $\Lotimes SM$ is $\gdim$-test  over $S$, then $M$ is   $\gdim$-test  over $R$.
\item\label{prop140825a1}
If $\Lotimes SM$ is  $\pd$-test  over $S$, then $M$ is $\pd$-test  over $R$.
\end{enumerate}
\end{thm}

\begin{proof}
\eqref{prop140825a2}
Assume that $\Lotimes SM$ is a  $\gkdim{\Lotimes SC}$-test complex over $S$.
To show that $M$ is a  $\gcdim$-test complex over $R$,  let $X\in\catdfb(R)$ such that
$\Lotimes MX\in\catdb(R)$.
By  flatness, the complexes
$\Lotimes SX$, $\Lotimes SM$, and $\Lotimes S{(\Lotimes MX)}$ are all in $\catdfb(S)$.
Moreover, we have the following isomorphisms
in $\catd(S)$:
$$\Lotimes[S]{(\Lotimes SM)}{(\Lotimes SX)}
\simeq\Lotimes{(\Lotimes SM)}X
\simeq\Lotimes S{(\Lotimes MX)}.$$
As $\Lotimes SM$ is a $\gkdim{\Lotimes SC}$-test complex over $S$,  we have $\gkdim{\Lotimes SC}_S(\Lotimes SX)<\infty$.
Using~\ref{ch150423a3}, we conclude that
$\gcdim_R(X)<\infty$, as desired.

\eqref{prop140825a2'}
This is the special case $C=R$ of part~\eqref{prop140825a2}.

\eqref{prop140825a1}
Argue as in part~\eqref{prop140825a2}, using~\cite[(1.5.3)]{avramov:rhafgd} in place of~\ref{ch150423a3}.
\end{proof}

Note that the conditions on $\vf$ in the following three items hold for the natural maps from $R$
to its completion $\comp R$ or to its henselization $R^{\text{h}}$.

\begin{rmk}\label{rmk150424a}
Let $\vf\colon R\to S$ be a flat local ring homomorphism, and
assume that the closed fibre $S/\m S$ is module-finite over $k$.
Let $M\in\catdfb(R)$ and $X\in\catdfb(S)$ such that
$\Lotimes[S] {(\Lotimes {S}M)}X\in\catdb(S)$.
Let $\x=x_1,\ldots,x_n\in\m$ be a  generating sequence for $\m$,
and set $K:=K^{S}(\x)$, the Koszul complex on $\x$ over $S$.
It follows that
$\Lotimes[S]K{(\Lotimes[S] {(\Lotimes {S}M)}X)}\in\catdb(S)$.
From the following isomorphisms
$$
\Lotimes[S]K{(\Lotimes[S] {(\Lotimes {S}M)}X)}
\simeq\Lotimes[S]{(\Lotimes[S]KX)}{(\Lotimes {S}M)}
\simeq\Lotimes[R]{(\Lotimes[S]KX)}{M}
$$
we conclude that $\Lotimes[R]{(\Lotimes[S]KX)}{M}\in\catdb(R)$.

Note that $\Lotimes[S]KX\in\catdfb(R)$. Indeed, we already have $\Lotimes[S]KX\in\catdb(R)$,
so it suffices to show that every homology module $\HH_i(\Lotimes[S]KX)$ is finitely generated over $R$.
We know that $\HH_i(\Lotimes[S]KX)$ is finitely generated over $S$.
Moreover, it is annihilated by $(\x)S=\m S$. Thus, it is a finitely generated $S/\m S$-module; since $S/\m S$ is module finite over $k$,
each $\HH_i(\Lotimes[S]KX)$ is finitely generated over $k$,
so it is finitely generated over $R$.
\end{rmk}

\begin{thm}\label{thm140825ax}
Let $\vf\colon R\to S$ be a flat local ring homomorphism, and let $M\in\catdfb(R)$.
Assume that the closed fibre $S/\m S$ is Gorenstein and module-finite over $k$.
\begin{enumerate}[\rm(a)]
\item \label{thm140825ax1}
The $R$-complex $M$ is $\gcdim$-test  if and only if $\Lotimes SM$ is $\gkdim{\Lotimes SC}$-test  over $S$.
\item \label{thm140825ax2}
The $R$-complex $M$ is  $\gdim$-test  if and only if $\Lotimes SM$ is   $\gdim$-test  over $S$.
\end{enumerate}
\end{thm}

\begin{proof}
\eqref{thm140825ax1}
One implication is covered by Theorem~\ref{prop140825a}\eqref{prop140825a2}.
For the reverse implication, 
assume that $M$ is a  $\gcdim$-test complex over $R$.
To show that $\Lotimes {S}M$ is a  $\gkdim{\Lotimes {S}C}$-test complex over $S$, 
let $X\in\catdfb(S)$ such that
$\Lotimes[S] {(\Lotimes {S}M)}X\in\catdb(S)$.
Let $\x=x_1,\ldots,x_n\in\m$ be a  generating sequence for $\m$,
and set $K:=K^{S}(\x)$.
By Remark~\ref{rmk150424a}
we have $\Lotimes[R]{(\Lotimes[S]KX)}{M}\in\catdb(R)$ and $\Lotimes[S]KX\in\catdfb(R)$.
Since $M$ is an $\gcdim$-test complex over $R$, one has $\gcdim_R(\Lotimes[S]KX)<\infty$.
It follows from Lemma~\ref{lem150421a} that $\gkdim{\Lotimes {S}C}_{S}(\Lotimes[S]KX)<\infty$.
We deduce from~\cite[(4.4)]{frankild:rrhffd} that
$\gkdim{\Lotimes {S}C}_{S}(X)<\infty$, as desired.

\eqref{thm140825ax2}
This is the special case $C=R$ of part~\eqref{thm140825ax1}.
\end{proof}

Here is one of our main results; see also Corollary \ref{cor140825a} and Theorem \ref{mainobscor}.

\begin{thm}\label{thm140825a}
Let $\vf\colon (R, \fm) \to (S, \fn)$ be a flat local ring homomorphism, and let $M\in\catdfb(R)$.
Assume the induced map $R/\fm \to S/\fm S$ is a finite field extension, i.e., we have $\fm S=\mathfrak{n}$
and the induced map $R/\fm \to S/\fm S$ is finite.
Then $M$ is a  $\pd$-test complex over $R$ if and only if $\Lotimes SM$ is a $\pd$-test complex over $S$.
\end{thm}

\begin{proof}
One implication is covered by Theorem~\ref{prop140825a}\eqref{prop140825a1}.
For the reverse implication, assume that $M$ is a  $\pd$-test complex over $R$.

Case 1: $S$ is complete.
To show that $\Lotimes {S}M$ is a  $\pd$-test complex over $S$, 
let $X\in\catdfb(S)$ such that
$\Lotimes[S] {(\Lotimes {S}M)}X\in\catdb(S)$.
Let $\x=x_1,\ldots,x_n\in\m$ be a  generating sequence for $\m$,
and set $K:=K^{S}(\x)$.
It follows that
$\Lotimes[S]K{(\Lotimes[S] {(\Lotimes {S}M)}X)}\in\catdb(S)$.
By Remark~\ref{rmk150424a}
we have $\Lotimes[R]{(\Lotimes[S]KX)}{M}\in\catdb(R)$ and $\Lotimes[S]KX\in\catdfb(R)$.
As $M$ is an $\pd$-test complex over $R$, we have $\pd_R(\Lotimes[S]KX)<\infty$.
It follows from~\cite[(2.5)]{avramov:holh} that $\pd_{S}(\Lotimes[S]KX)<\infty$, and~\cite[(1.5.3)]{avramov:rhafgd} implies that
$\pd_{S}(X)<\infty$.

Case 2: the general case.
Case 1 implies that
$\Lotimes {\comp S}M$ is a  $\pd$-test complex over $\comp S$,
i.e.,
$\Lotimes[S]{\comp S}{(\Lotimes SM)}$ is a  $\pd$-test complex over $\comp S$.
So, $\Lotimes SM$ is a  $\pd$-test complex over $S$, 
by Theorem~\ref{prop140825a}\eqref{prop140825a1}, as desired.
\end{proof}

The next example, from discussions with Ryo Takahashi,
shows that the hypothesis $\fm S=\mathfrak{n}$ is necessary for the conclusion of Theorem \ref{thm140825a}. 

\begin{ex} \label{exRyo} Let $k$ be a field, $R=k$ and $S=k[\![y]\!]/(y^2)$. Then the natural map $R \to S$ is a finite free map since $S$ is free over $R$ with $R$-basis $\{1,y\}$. Let $M=R$. Then, since $R$ is regular, $M$ is a $\pd$-test module over $R$. However $N=M\otimes_{R}S=S$ is not a $\pd$-test module over $S$ since 
$S$ is not regular; see~\ref{obs150421a}.
\end{ex}

On the other hand, we do not know whether or not having a regular closed fibre in Theorem \ref{thm140825a} is sufficient, as we note next.

\begin{que}
Let $\vf\colon (R, \fm) \to (S, \fn)$ be a flat local ring homomorphism, and let $M\in\catdfb(R)$.
Assume that $S/\fm S$ is regular.
If $M$ is a  $\pd$-test complex over $R$, then must $\Lotimes SM$ be a $\pd$-test complex over $S$?
\end{que}

The next corollary answers Question~\ref{q1}\eqref{item140831b}.

\begin{cor}  \label{cor140825b}
Let $M$ be an $R$-module, and set $\comp C:=\Lotimes{\comp R}C$.
\begin{enumerate}[\rm(a)]
\item\label{cor140825b1}
The module $M$ is $\gcdim$-test  over $R$ if and only if $\comp M$ is   $\gkdim{\comp C}$-test  over $\comp R$.
\item\label{cor140825b1'}
The module $M$ is  $\gdim$-test  over $R$ if and only if $\comp M$ is   $\gdim$-test over $\comp R$.
\item\label{cor140825b2}
The module $M$ is   $\pd$-test  over $R$ if and only if $\comp M$ is $\pd$-test  over $\comp R$.
\end{enumerate}
\end{cor}

\begin{proof}
Since $\Lotimes{\comp R}M\simeq\comp M$ in $\catdfb(\comp R)$, the desired conclusions follow from
Theorems~\ref{thm140825ax} and~\ref{thm140825a}.
\end{proof}

The next  corollary answers Question~\ref{q1}\eqref{item140831a}. 
We are able to improve this result significantly in the next section; see Theorem~\ref{mainobscor} and the subsequent paragraph.

\begin{cor} \label{cor140825a} Let $M$ be a  $\pd$-test module over $R$. If $\gdim_R(M)<\infty$, then $R$ is Gorenstein.
\end{cor}

\begin{proof}
Corollary~\ref{cor140825b}\eqref{cor140825b2} says that $\comp M$ is a $\pd$-test module for $\comp R$,
with $\gdim_{\comp R}(\comp M)<\infty$ by~\ref{ch150423a3}.
Using~\cite[(1.3)]{CDtest}, we conclude that $\comp R$ is Gorenstein, hence so is $R$.
\end{proof}

We end this section by building a 
module that is $\gdim$-test but not $\pd$-test;
see Example~\ref{ex140816a}.

\begin{prop}\label{prop140816a}
Let $\vf\colon (A,\n,F)\to R$ be a flat local ring homomorphism, and set $N:=R/\n R\simeq \Lotimes[A]R{F}$.
\begin{enumerate}[\rm(a)]
\item \label{prop140816a2}
The $R$-module $N$ is Tor-rigid, i.e., for any finitely generated $R$-module $M$, if $\Tor iMN=0$ for some $i\geq 1$, then $\Tor jMN=0$ for all $j\geq i$.
\item \label{prop140816a1}
Let $B$ be a semidualizing $A$-complex, and set $C:=\Lotimes[A]RB$.
If $R/\n R$ is Gorenstein, then $N$ is a $\gcdim$-test complex over $R$.
\item \label{prop140816a1'}
If $R/\n R$ is Gorenstein, then $N$ is a $\gdim$-test complex over $R$.
\end{enumerate}
\end{prop}

\begin{proof}
First, note that we have $N:=R/\n R\cong\Otimes[A]R{(A/\n)}=\Otimes[A]RF\simeq \Lotimes[A]R{F}$ since $R$ is flat over $A$.
Furthermore, for every $R$-complex $M$, there are isomorphisms
\begin{equation}\label{eq150424a}
\Lotimes MN\simeq\Lotimes M{(\Lotimes[A] RF)}\simeq\Lotimes[A]MF.
\end{equation}
In particular, one has $\Tor iMN\cong\Tor[A]iMF$ for all $i$. 

\eqref{prop140816a2}
Let $M$ be a finitely generated $R$-module.
From~\cite[(22.3)]{Mat}, if $\Tor[A] 1MF=0$, then $M$ is flat over $A$, so we have $\Tor[A] jMF=0$ for all $j\geq 1$.
More generally, by dimension-shifting, if $\Tor[A] iMF=0$ for some $i\geq 1$, then  we have $\Tor[A] jMF=0$ for all $j\geq i$.
Thus, the isomorphism from the previous paragraph implies the desired Tor-rigidity.

\eqref{prop140816a1}
Corollary~\ref{cor140825b}\eqref{cor140825b1} shows that it suffices to show that $\comp N$ is a $\gkdim{\Lotimes {\comp R}C}$-test complex over $\comp R$.
Note that the induced map $\comp\vf\colon\comp A\to \comp R$ is flat and local with Gorenstein closed fibre $\comp{R/\n R}$.
Also, there are isomorphisms
in $\catd(\comp R)$:
$$\Lotimes[\comp A]{\comp R}{(\Lotimes[A]{\comp A}{B})}
\simeq \Lotimes[A]{\comp R}{B}
\simeq \Lotimes[R]{\comp R}{(\Lotimes[A]{R}{B})}
\simeq \Lotimes[R]{\comp R}{C}.
$$
Thus, we may replace  $\vf$ with the induced map $\comp\vf$ to assume for the rest of the proof that $A$ and $R$ are complete.
Let $D^{A}$ be a dualizing $A$-complex; see~\ref{ch150423c1}.
Then the $R$-complex $D^{R}:=\Lotimes[A]{R}{D^{A}}$ is dualizing for $R$, by~\ref{ch150423c2}.
Set $\da B:=\Rhom[A]{B}{D^A}$ and $\da C:=\Rhom C{D^R}$, noting that $\da C\simeq \Lotimes[A]{R}{\da B}$
by~\cite[proof of (5.10)(*)]{LarsSD}.

Let $M\in\catdfb(R)$ such that $\Lotimes MN\in\catdb(R)$.
We need to show  $\gcdim_R(M)<\infty$
i.e.,  that $M\in\cata_{\da C}(R)$; see~\ref{ch150423b1}.
By~\eqref{eq150424a}, the complex $\Lotimes[A]MF$ is in $\catdb(R)$.
Since $M$ is in $\catdfb(R)$, we conclude that $\fd_A(M)<\infty$ by~\cite[(5.5.F)]{avramov:hdouc}.
It follows from~\cite[(4.4)]{LarsSD} that $M\in\cata_{\da B}(A)$, so $M\in\cata_{\Lotimes[A]{R}{\da B}}(R)=\cata_{\da C}(R)$ by~\cite[(5.3.a)]{LarsSD}.

\eqref{prop140816a1'}
This is the special case $B=A$, hence $C=R$, of part~\eqref{prop140816a1}.
\end{proof}

\begin{ex}\label{ex140816a}
Let $k$ be a field. Consider the finite-dimensional local $k$-algebras
$A:=k[y,z]/(y,z)^2$ and 
$$R:=k[x,y,z]/(x^2,y^2,z^2,yz)\cong A[x]/(x^2).$$
Notice that $R$ is free over $A$, hence flat.
Also, the natural map $A\to R$ is local with Gorenstein closed fibre $R/\n R\cong k[x]/(x^2)$;
here, as in Proposition~\ref{prop140816a}, we let $\n$ denote the maximal ideal of $A$.
Since the assumptions of Proposition~\ref{prop140816a} are satisfied, the $R$-module
$$N=\Otimes[A]  Rk=\Otimes[A]  R{(A/(y,z)A)}\cong R/(y,z)R$$
is Tor-rigid and $\gdim$-test over $R$. 
Furthermore, from~\cite[(4.1)]{CDtest}, we know that $N$ is not $\pd$-test.

Since $A$ is artinian and local, the injective hull $D^{A}=E_R(k)$ is a dualizing $A$-module.
Thus, the $R$-module $D^{R}:=\Otimes[A]R{D^A}\simeq\Lotimes[A]{R}{D^{A}}$ is dualizing for $R$, by~\ref{ch150423c2}.
We conclude by showing that $D^R$ is also a $\gdim$-test module that is not $\pd$-test and, moreover, is not Tor-rigid.

Note that $A$ has length 3 and type 2.
From this, we construct an exact sequence over $A$ of the following form:
\begin{equation}\label{eq150424y}
0\to k^3\to A^2\to D^A\to 0.
\end{equation}
Indeed, the condition $\operatorname{type}(A)=2$ says that $D^A$ is minimally generated by 2 elements.
Let $A^2\xra{p}D^A\to 0$ be a minimal presentation, and consider the corresponding short exact sequence
$$0\to \operatorname{Ker}(p)\to A^2\xra{p}D^A\to 0.$$
From the minimality of the presentation, it follows that $\operatorname{Ker}(p)\subseteq\n A^2$.
Since $\n^2=0$, we conclude that $\operatorname{Ker}(p)$ is a $k$-vector space, so we need only verify that
$\len_A(\operatorname{Ker}(p))=3$.
This equality follows from  the additivity of length, via the condition $\len_A(D^A)=\len(A)=3$.

Since $R$ is flat over $A$, we apply the base-change functor $\Otimes[A]R-$ to the sequence~\eqref{eq150424y}
to obtain the next exact sequence over $R$:
\begin{equation*}
0\to N^3\to R^2\to D^R\to 0.
\end{equation*}
For any $R$-module $M$, the associate long exact in $\Tor iM-$ shows that we have
$\Tor {i}M{D^R}\cong\Tor {i-1}MN^3$
for all $i\geq 2$. In particular,  we have 
\begin{equation}\label{eq150424b}
\text{$\Tor {i}M{D^R}=0$ if and only if $\Tor {i-1}MN=0$, for $i\geq 2$.}
\end{equation}

Claim: $D^R$ is $\gdim$-test over $R$.
To show this, let $M$ be a finitely generated $R$-module such that $\Tor iM{D^R}=0$ for $i\gg 0$.
The display~\eqref{eq150424b} implies that $\Tor iMN=0$ for $i\gg 0$.
Since $N$ is  $\gdim$-test  over $R$, we have $\gdim_R(M)<\infty$, as desired.

Claim: $D^R$ is not $\pd$-test over $R$.
To show this, suppose by way of contradiction that $D^R$ is $\pd$-test over $R$.
We  show that $N$ is $\pd$-test, contradicting~\cite[(4.1)]{CDtest}. 
Let $M$ be a finitely generated $R$-module such that $\Tor iM{N}=0$ for $i\gg 0$.
The display~\eqref{eq150424b} implies that $\Tor iM{D^R}=0$ for $i\gg 0$.
Since $D^R$ is $\pd$-test, we have $\pd_R(M)<\infty$. Thus, $N$ is $\pd$-test, giving the advertised contradiction, and establishing the claim.

Claim: If $M$ is a finitely generated $R$-module such that $\Tor iM{D^R}=0$ for some $i\geq 2$, then
$\Tor jM{D^R}=0$ for all $j\geq i$. (This shows that $D^R$ is almost Tor-rigid.)
Since $N$ is Tor-rigid by Proposition~\ref{prop140816a}\eqref{prop140816a2}, 
this follows from~\eqref{eq150424b}.

Claim: $D^R$ is not Tor-rigid over $R$.
To this end, we follow a construction of~\cite[Chapter 3]{EG} and 
build an $R$-module $L$ such that $\Tor 1L{D^R}=0$ and $\Tor iL{D^R}\neq 0$ for all $i\geq 2$. Let $f_1,f_2$ be a minimal generating sequence for $\Hom[A]kA\cong k^2$.
For instance, $f_1(1)=y$ and $f_2(1)=z$ will work here.
Define $u\colon k\to A^2$ by the formula $a\mapsto (f_1(a),f_2(a))$.
Since $k$ is simple, the non-zero map $u$ is a monomorphism.
Let $M:=\coker(u)$. The long exact sequence in $\Ext[A]{}-A$ 
associated to the sequence
\begin{equation}\label{eq150424x}
0\to k\xra u A^2\to M\to 0
\end{equation}
shows  that
$\Ext[A] 1MA=0$.

Set $L:=\Otimes[A]RM$.
To make things concrete, if one uses the specific functions $f_1,f_2$ suggested in the previous paragraph, then 
$M$ has the following minimal free presentation over $A$
$$A\xra{\left(\begin{smallmatrix}y\\ z\end{smallmatrix}\right)}A^2\to M\to 0$$
so $L$ has the following minimal free presentation over $R$
$$R\xra{\left(\begin{smallmatrix}y\\ z\end{smallmatrix}\right)}R^2\to L\to 0.$$
Now, flat base-change implies that
$$\Ext 1LR \cong \Otimes[A]R{\Ext[A] 1MA}=0.$$
Thus, we have
$$\Hom{\Tor 1{D^R}L}{D^R}\cong\Ext 1LR=0.$$
The fact that $D^R$ is faithfully injective over $R$ implies that $\Tor 1{D^R}L=0$.

Since $A$ is not regular, we have $\pd_A(k)=\infty$ and $\Tor[A]i{k}k\neq 0$ for all $i>0$. Therefore, by \eqref{eq150424x}, 
we conclude that $\pd_A(M)=\infty$. In particular, this implies that
$\Tor[A]1{k}M\neq 0$. So it follows from \eqref{eq150424y} and \eqref{eq150424x} that
$$\Tor[A]i{D^A}M \cong \Tor[A]{i-1}{k^3}M\cong \Tor[A]{i-2}{k^3}k \cong  (\Tor[A]{i-2}{k}k)^3 \neq 0 \text{ for all } i>2$$
and
$$\Tor[A]2{D^A}M \cong \Tor[A]{1}{k^3}M \cong (\Tor[A]{1}{k}M)^3 \neq 0.$$
Thus flat base-change implies that
$$\Tor i{D^R}L\cong\Otimes[A]R{\Tor[A]i{D^A}M}\neq 0 \text{ for all } i\geq 2.$$ 
This completes the claim and the example.
\end{ex}

\section{Detecting the Dualizing and Gorenstein Properties}
\label{sec150429b}

Our next result yields both Theorems~\ref{mainobscor} and~\ref{cor2} highlighted in the introduction.
Note that condition~\eqref{mainobs2} in this result does not assume
\emph{a priori} that $R$ has a dualizing complex; however, the result shows that this condition implies
that $R$ has a dualizing complex.

\begin{thm} \label{mainobs}   
Let $B$, $C$ be semidualizing $R$-complexes.
Let $M \not\simeq 0$ be a  $\gkdim B$-test $R$-complex such that $\Rhom MC \in \catdb(R)$. 
Assume that one of the following conditions holds:
\begin{enumerate}[\rm(1)]
\item\label{mainobs1}   
The ring $R$ has a dualizing complex $D$, and $\gkdim{\da B}_R(C)<\infty$ where
$\da B:=\Rhom BD$, or
\item\label{mainobs2} 
One has $C\in\catab(R)$.
\end{enumerate}  
Then $\Lotimes BC$ is dualizing for $R$.
\end{thm}

\begin{proof} 
\eqref{mainobs1}
Assume that $R$ has a dualizing complex $D$, and $\gkdim{\da B}_R(C)<\infty$. 
Set $\da C:=\Rhom CD$ which is semidualizing by~\ref{ch150423c1}.
Let $\x=x_1, \ldots, x_{n}\in\m$ be a generating sequence for $\m$, and consider the Koszul complex $K:=K^R(\x)$. Set $X:=\Lotimes K{\Rhom CE}$ where $E=E_R(k)$ is the injective hull of $k$.

We note that $X$ is in $\catdfb(R)$. Indeed, 
the complex $\Rhom KC$ is homologically finite since $C$ is. The total homology
module $\HH(\Rhom KC)$ is annihilated by $(\x)R=\m$, so it is a finite dimensional vector space over $k$. 
By Matlis duality, the total homology module of $\Rhom {\Rhom KC}E$ is also a 
finite dimensional vector space over $k$, so 
we have
\begin{equation}\label{eq140825b}
X=\Lotimes K{\Rhom CE}\simeq\Rhom {\Rhom KC}E\in\catdfb(R)
\end{equation}
by Hom-evaluation~\cite[(4.4)]{avramov:hdouc}.

The assumption $\Rhom MC \in \catdb(R)$ implies that
$$\Lotimes M{\Rhom CE}\simeq\Rhom{\Rhom MC}E\in \catdb(R)$$
again by Hom-evaluation~\cite[(4.4)]{avramov:hdouc}.
From this, we conclude that
$$\Lotimes MX=\Lotimes M{(\Lotimes K{\Rhom CE})}\simeq\Lotimes K{(\Lotimes M{\Rhom CE})}\in \catdb(R).$$
Since $M$ is a  $\gkdim B$-test complex, this implies that
$\gkdim{B}_R(X)<\infty$.
From~\ref{ch150423b1}
we conclude that 
$X\in\cata_{\da B}(R)$,
i.e., we have $\Rhom {\Rhom KC}E\in\cata_{\da B}(R)$ by~\eqref{eq140825b}.
Since $E$ is faithfully injective, argue as in the proof of~\cite[(3.2.9)]{Chr} to conclude that
$\Rhom KC\in\catb_{\da B}(R)$; see also~\cite[4.14]{sather:afc}.
By assumption, we have $C\in\catdfb(R)$,
so Lemma~\ref{fact140901a}\eqref{fact140901a2} shows that $C\in\catb_{\da B}(R)$.

By assumption, we have $\gkdim{\da B}_R(C)<\infty$,
so~\cite[(1.3)]{frankild:rbsc} implies that $\da B\in\catb_{C}(R)$.
We conclude from~\cite[(1.4) and (4.10.4)]{frankild:rbsc} that $C$ and $\da B$ are isomorphic up to a shift in $\catd(R)$.
Apply a shift to $C$ to assume that $C\simeq\da B$.
From~\cite[(4.4)]{LarsSD} we have
$D\simeq\Lotimes B{\da B}\simeq\Lotimes BC$, as desired.

\eqref{mainobs2}
Assume now that $C\in\catab(R)$.
The completion $\comp R$ has a dualizing complex $D^{\comp R}$, by~\ref{ch150423c1},
and the complexes $\Lotimes {\comp R}B$ and $\Lotimes {\comp R}C$ are semidualizing over $\comp R$, by~\ref{ch150423c2}.
We have $\Lotimes{\comp R}M \not\simeq 0$ by faithful flatness of $\comp R$, and
the complex $\Lotimes{\comp R}M$ is $\gkdim{\Lotimes{\comp R}B}$-test
by Theorem~\ref{thm140825ax}\eqref{thm140825ax1}.
Also, by faithful flatness, the condition $\Rhom MC \in \catdb(R)$
implies that
$$\Rhom[\comp R]{\Lotimes{\comp R}M}{\Lotimes{\comp R}C}\simeq\Lotimes{\comp R}{\Rhom MC} \in \catdb(\comp R)$$
by~\cite[(1.9.a)]{frankild:rrhffd}.

With $(-)^\dagger:=\Rhom[\comp R]-{D^{\comp R}}$, we have
an isomorphism $\Lotimes{\comp R}B\simeq(\Lotimes{\comp R}B)^{\dagger\dagger}$
by~\ref{ch150423a1}.
In addition, from~\cite[(5.8)]{LarsSD}, the assumption $C\in\cata_B(R)$ implies that
we have $\Lotimes{\comp R}C\in\cata_{\Lotimes{\comp R}B}(\comp R)=\cata_{(\Lotimes{\comp R}B)^{\dagger\dagger}}(\comp R)$.
We conclude from~\ref{ch150423b1} that $\gkdim{(\Lotimes{\comp R}B)^\dagger}_{\comp R}(\Lotimes{\comp R}C)<\infty$.

It follows that condition~\eqref{mainobs1} is satisfied over $\comp R$.
Thus, the $\comp R$-complex 
$$
\Lotimes[\comp R]{(\Lotimes{\comp R}B)}{(\Lotimes{\comp R}C)}
\simeq\Lotimes{(\Lotimes{\comp R}B)}{C}
\simeq\Lotimes{\comp R}{(\Lotimes BC)}
$$
is dualizing for $\comp R$.
Note that the condition $C\in\cata_B(R)$ implies by definition that
$\Lotimes BC\in\catdb(R)$.
Thus, the condition $B,C\in\catdf_{\text{b}}(R)$ implies that $\Lotimes BC\in\catdfb(R)$.
From this, the fact that $\Lotimes{\comp R}{(\Lotimes BC)}$ is dualizing for $\comp R$
implies that 
$\Lotimes BC$ is dualizing for $R$, by~\cite[(5.1)]{avramov:lgh}, as desired.
\end{proof}

We now give several consequences of Theorem~\ref{mainobs}.
Compare the next result to~\cite[(3.4)]{CDtest}.

\begin{cor} \label{mainobsx} Let $M$ be a  $\gdim$-test $R$-complex.  
Let $C$ be a semidualizing $R$-complex such that that $\Rhom MC \in \catdb(R)$. Then $C$ is dualizing for $R$.
\end{cor}

\begin{proof} 
By~\ref{obs150421a}, we assume  that $M\not\simeq 0$.
The desired conclusion follows from  Theorem~\ref{mainobs}  with $B=R$,
once we note that $C\in\catdb(R)=\cata_R(R)$.
\end{proof}

\begin{cor} \label{mainobscorq} 
Let $M$ be a  $\gdim$-test $R$-complex such that $\Rhom MR \in \catdb(R)$. Then $R$ is Gorenstein. 
\end{cor}

\begin{proof}
Use $C=R$ in Corollary~\ref{mainobsx}.
\end{proof}

\begin{thm} \label{mainobscor} Let $M$ be a  $\gdim$-test $R$-module such that $\Ext iMR=0$ for $i\gg 0$, e.g., with $\gdim_R(M)<\infty$. Then $R$ is Gorenstein.
\end{thm}

\begin{proof}
This is immediate from Corollary~\ref{mainobscorq}.
\end{proof}

We note that the hypotheses of Theorem~\ref{mainobscor} are weaker than those in Corollary~\ref{cor140825a}.
Indeed, Example~\ref{ex140816a} above exhibits a $\gdim$-test module that is not a $\pd$-test module.
Furthermore, as we noted in~\ref{obs140825a3}, there exist examples of finitely generated modules $L$
such that  $\Ext iLR=0$ for all $i\geq 1$ and $\gdim_R(L)=\infty$.

\begin{cor} \label{cor140901a} 
Let $C$ be a semidualizing $R$-complex. If $C$ is $\gdim$-test over $R$, then $C$ is dualizing for $R$.
\end{cor}

\begin{proof} 
This follows from Corollary~\ref{mainobsx}
since $\Rhom CC\simeq R\in\catdb(R)$.
\end{proof}

\begin{disc}
\label{disc140901a}
In light of Corollary~\ref{cor140901a},
it is worth noting that there are rings with semidualizing complexes that are not dualizing and that have infinite projective dimension.
(In particular, these complexes are neither $\pd$-test nor $\gdim$-test by~\ref{obs140825a1} and Corollary~\ref{cor140901a}.)
The first examples were constructed (though not published) by Foxby.
See also~\cite[(7.8)]{LarsSD} and~\cite{sather:divisor,sather:lbnsc}.

It is also worth noting that the converse of Corollary~\ref{cor140901a} fails in general by~\ref{obs140825a3}.
\end{disc}

\begin{cor} \label{cor2z} Let $I\subseteq R$ be an integrally closed ideal such that $\depth(R/I)=0$, e.g., such that $I$ is $\fm$-primary, and let $C$ be a semidualizing $R$-complex. 
Then $C$ is dualizing for $R$ if and only if $\gcdim_R(R/I)<\infty$.
\end{cor}

\begin{proof} Note that $I$ is a $\pd$-test by Lemma~\ref{lem150424a}, and apply~\ref{ch150423a1} and Corollary~\ref{mainobsx}.
\end{proof}

Recall that the next result has been initially obtained by Goto and Hayasaka~\cite[(1.1)]{GH} under some extra conditions; see also \cite[(2.2)]{GH2}.

\begin{thm} \label{cor2} 
Let $I\subseteq R$ be an integrally closed ideal with $\depth(R/I)=0$, e.g., such that $I$ is $\fm$-primary. Then $R$ is Gorenstein if and only if $\gdim_R(R/I)<\infty$.
\end{thm}

\begin{proof} Apply Corollary~\ref{cor2z} with $C=R$,
or use Theorem~\ref{mainobscor} with Lemma~\ref{lem150424a}.
\end{proof}

We finish this section by giving two examples that show the integrally closed and depth hypotheses of Theorem \ref{cor2} are 
necessary:

\begin{ex} \label{ex-cor2-1} Let $k$ be a field and let $R=k[\![x,y,z]\!]/(x^2,y^2,z^2,yz)$, as in Example \ref{ex140816a}. Then, since $R$ is Artinian, each proper ideal of $R$ is $\fm$-primary but is not integrally closed; see \cite[(1.1.3)(3)]{CS}. In particular the principal ideal $I$ of $R$ generated by $x$ is $\fm$-primary but not integrally closed. Note that $R$ is not Gorenstein. Also, we have $\gdim(I)=0$
because the fact that $R$ is of the form $A[x]/(x^2)$ implies that $I$ has a complete resolution $\cdots\xra xR\xra xR\xra x\cdots$.
\end{ex}

\begin{ex} \label{ex-cor2-2} Let $k$ be a field, set $S:=k[\![x_1, x_2, x_3, y_1, y_2, y_3]\!]$, and let $I$ be the ideal of $S$ generated by $x_1y_2-x_2y_1$, $x_1y_3-x_3y_1$, and $x_2y_3-y_2x_3$. Set $R :=S/I$. Then $R$ is a four-dimensional normal Cohen-Macaulay domain that is not Gorenstein; see \cite[Theorem (a)]{Goto}. Let $0\neq f\in \fm$. Then the ideal $fR$ of $R$ generated by $f$ is integrally closed; see \cite[(1.5.2)]{CS}. Furthermore  $\pd_R(R/fR)=\gdim_R(R/fR)=1$ and $\depth(R/fR)=3$.
\end{ex}

\section*{Acknowledgments}
Parts of this work were completed when Celikbas visited  North Dakota State University in November 2013, and when Sather-Wagstaff visited the University of Connecticut in April 2015.  We are grateful for the kind hospitality and the generous support of the NDSU and UConn Mathematics Departments. We are also grateful to Jerzy Weyman for supporting Sather-Wagstaff's visit, and to Irena Swanson and Ryo Takahashi for helpful feedback on this work. We thank Naoki Taniguchi and Shiro Goto for pointing out Example \ref{ex-cor2-2} to us. We also thank the referee for his/her valuable corrections and suggestions on the manuscript.

\bibliographystyle{amsplain}
\providecommand{\bysame}{\leavevmode\hbox to3em{\hrulefill}\thinspace}
\providecommand{\MR}{\relax\ifhmode\unskip\space\fi MR }
\providecommand{\MRhref}[2]{%
  \href{http://www.ams.org/mathscinet-getitem?mr=#1}{#2}
}
\providecommand{\href}[2]{#2}

\end{document}